\documentclass[11pt]{amsart}
\usepackage{multirow,latexsym,mathptmx,graphicx,fancyhdr, enumerate}
\usepackage{amsfonts,amscd,amssymb,amsmath,amsthm,mathrsfs,xcolor,lscape,ifthen,mathtools}
\usepackage[english]{babel}
\usepackage[all,cmtip]{xy}
\usepackage{tikz}
\usepackage{tikz-cd,xr,multicol,microtype}
\usepackage{tkz-euclide}
\usetikzlibrary{matrix}
\usetikzlibrary{calc}
\usepackage{todonotes}
\usepackage{ucs}
\usepackage[utf8]{inputenc}
\usepackage[T1]{fontenc} 
\usepackage[thinlines]{easytable}
\usepackage{color}
\usepackage[colorlinks,linkcolor=blue,anchorcolor=blue,citecolor=blue,backref=page]{hyperref}
\usepackage{hyperref}
\usepackage{phaistos}
\usepackage{microtype}

\newtheorem{thm}{Theorem}[section]
\newtheorem{lem}[thm]{Lemma}
\newtheorem{question}[thm]{Question}

\newtheorem{defn}[thm]{Definition}
\newtheorem{cor}[thm]{Corollary}
\newtheorem{prop}[thm]{Proposition}

\newtheorem{example}[thm]{Example}

\newtheorem{remark}[thm]{Remark}

\newcommand{\mr}[1]{\mathrm{#1}}
\newcommand{\C}{\mathbb{C}}

\newcommand{\R}{\mathbb{R}}

\newcommand{\h}{\mathbb{H}}
\newcommand{\X}{\mathbb{X}}
\newcommand{\Y}{\mathbb{Y}}
\newcommand{\Z}{\mathbb{Z}}
\newcommand{\Q}{\mathbb{Q}}
\newcommand{\n}{\mathfrak{n}}

\newcommand{\m}{\mathrm{M}}

\newcommand{\Cu}{\mathfrak{C}}
\newcommand{\rG}{\mathrm{G}}
\newcommand{\rH}{\mathrm{H}}

\newcommand{\tq}{\widetilde{q}}

\newcommand{\Ii}{\mathrm{i}}

\def\lim{\mathrm{lim}}
\newcommand{\g}{\gamma}

\def\cu{{\mathfrak{c}}}

\def\exp{\mathrm{exp}}

\newcommand{\tmt}[4]{\left({#1\atop #3}{#2\atop #4}\right)}
\newcommand{\cmt}[2]{\left({#1\atop #2}\right)}

\newcommand{\beq}{\begin{equation}}
\newcommand{\eeq}{\end{equation}}

\newcommand\norm[1]{\left\lVert#1\right\rVert}
\newcommand{\Sp}{R(t_{\cu})}

\begin{document}
\date{\today}
\title[Lifting of vector-valued automorphic forms]{Lifting of vector-valued automorphic  forms}
\author{Jitendra Bajpai}
\noindent\address{Max-Planck-Institut f\"ur Mathematik, 53111 Bonn, Germany.}
\email{jitendra@mpim-bonn.mpg.de}
\curraddr{Department of Mathematics, Kiel University, 24118 Kiel, Germany}
\email{jitendra@math.uni-kiel.de}
\author{Subham Bhakta}
\address{Mathematisches Institut, Georg-August-Universit\"at G\"ottingen, 37073 G\"ottingen, Germany.} 
\email{subham.bhakta@mathematik.uni-goettingen.de }
\curraddr{Department of Mathematics, Indian Institute of Science Education and Research, Thiruvananthapuram, Kerala 695551, India}
\email{subham1729@iisertvm.ac.in}


\maketitle


\begin{abstract}
A recent study by the first author~\cite{Bajpai2019} demonstrated that admissible vector-valued automorphic forms have admissible lifts. In this article, we study the specific case of logarithmic vector-valued automorphic forms, exploring their lifts and providing explicit computations of their Fourier coefficients.
\end{abstract}

\section{Introduction}
Let $\bar{\rG} \subset \text{PSL}_2(\R)$ be a Fuchsian group of the first kind, and $\rho:\bar{\rG} \to\text{GL}_m(\mathbb{C})$ be any  finite-dimensional representation. A vector-valued automorphic form (vvaf) of $\bar{\rG}$ of weight $k\in 2\Z$  is a $\C^{m}$-valued meromorphic function $\X(\tau)$ on the complex upper half plane $\h$ which has a certain functional and cuspidal behaviour with respect to $\rho$. For details and explanation, see Section~\ref{sec:vvaf}. The theory of vector-valued automorphic forms has many important applications not only in mathematics but also in physics. For more details on the importance of their theory, see the introduction of~\cite{Bajpai2019,Gannon1}. 

We call a representation $\rho$ admissible if $\rho(\gamma)$ is diagonalizable for every parabolic element $\gamma \in \bar{\rG}.$ In that case, we call the associated vvaf an admissible vvaf. Now if we take $\bar{\rH}$ to be a finite index subgroup of $\bar{\rG},$ then any representation $\rho$ of $\bar{\rH}$ can be induced to a representation $\widetilde{\rho}$ of $\bar{\rG}.$ It turns out that, a vvaf $\X(\tau)$ of $\bar{\rH}$ associated to $\rho$ can also be induced to a vvaf $\widetilde{\X}(\tau)$ of $\bar{\rG}$ associated to $\widetilde{\rho}$, see Section~\ref{sec:liftinduction}. In short, vvaf $\widetilde{\X}(\tau)$, which is a lift of vvaf $\X(\tau)$ will be referred to as \emph{lifted form} throughout the article. In~\cite{Bajpai2019}, the first author showed that the induction of an admissible representation is admissible. Hence, the induced function $\widetilde{\X}(\tau)$ is also a weakly holomorphic admissible vvaf associated with the induced representation. We are primarily interested in the case when $\rho$ is not admissible. We furthermore assume that any eigenvalue of the image by $\rho$ of a parabolic element is unitary, a technical condition that will be explained in Section~\ref{sec:results}. A fundamental example of a non-admissible $\rho$ verifying the unitary condition is the defining representation $\Gamma\rightarrow\mathrm{GL}_2(\mathbb{C})$. Vector-valued automorphic forms satisfying both of these conditions are called \emph{logarithmic vector-valued automorphic forms}. Knopp and Mason~\cite{KM2011, KM2012} and Gannon~\cite{Gannon1} studied this case when $\bar{\rG}$ is the modular group. Recently, the authors studied them for any Fuchsian group of the first kind in~\cite{BBF}. 

In this article, we discuss what happens to logarithmic vector-valued automorphic forms after lifting them. We shall work with the Fuchsian groups of the first kind $\bar{\rG}$, and group $\rG \subset \mr{SL}_2(\R)$, where $\bar{\rG}$ denotes the image of $\rG$ by the natural map $\mr{SL}_2(\R) \longrightarrow \mr{PSL}_2(\R)$. To prepare ourselves for further discussions, we provide a quick example of a logarithmic vvaf.

\begin{example}
   An elementary example of logarithmic vvaf is $\X(\tau)=\cmt{\tau}{1}$ of weight $-1$ for the full modular group $\mr{SL}_2(\Z)=\left<  s=\tmt{0}{-1}{1}{\,\, 0}, t=\tmt{1}{1}{0}{1}\right>$, with respect to the representation $\rho:\mr{SL}_2(\Z)\longrightarrow \mathrm{GL}_2(\mathbb{C})$, which is the defining representation of $\mr{SL}_2(\Z)$, that is $\rho(s)=s$ and $\rho(t)=t$.
\end{example}

\subsection{Fuchsian groups}\label{sec:fg}
Let $\h$ be the upper half plane and $\h^{*}=\h \cup \R \cup \{\infty \}$ be the extended upper half plane. It is well known that, $\mr{PSL}_2(\R)$ acts on $\h^{*}$. For any $\g=\overline{\tmt{a}{b}{c}{d}} \in \mr{PSL}_2(\R)$, the action of $\g$ is defined by the M\"obius action $\g \cdot \tau= \frac{a\tau+b}{c\tau+d}.$

A Fuchsian group is a discrete subgroup $\bar{\rG}$ of PSL$_2(\R)$ for which $\bar{\rG}\backslash \h$ is a Riemann surface with finitely many punctures. Because of the discreteness, a fundamental domain exists for the action on $\h$. When it has a finite area, the corresponding Fuchsian group is of the first kind.

An element of $\g\in \mr{PSL}_2(\R)$ is called a parabolic element, if $|\text{tr}(\g)|=2.$ A point $\tau \in \h^{*}$ is said to be a fixed point of $\g \in \mr{PSL}_2(\R)$ if $\g \cdot \tau=\tau.$ If $\g=\overline{\tmt{a}{b}{c}{d}}$ is a parabolic element then its fixed point $\tau= \frac{a \mp 1}{c}$ when $a+d=\pm 2$ and $c\neq 0$, in addition $\tau=\infty$ when $c=0$. Moreover, we call an element $\tmt{a}{b}{c}{d}\in \text{SL}_2(\R)$ as parabolic, if the corresponding element $\overline{\tmt{a}{b}{c}{d}}\in \mathrm{PSL}_2(\R)$ is parabolic.

Let $\bar{\rG}$ be a subgroup of $\mr{PSL}_2(\R)$. A point $\tau \in \h^{*}$ is called a cusp of $\bar{\rG}$ (equivalently, of $\rG$) if it is fixed by some nontrivial parabolic element of $\bar{\rG}$ (equivalently, of $\rG$). Let $\Cu_{\rG}$ denote the set of all cusps of $\rG$ and we define $\h_{\rG}^{*}=\h \cup \Cu_{\rG}$ to be the extended upper half plane of $\rG$. For example: if $\rG=\mr{SL}_2(\R)$ then $\Cu_{\rG}=\R \cup \{ \infty \}$ and
 if $\rG=\mr{SL}_2(\Z)$ then $\Cu_{\rG}=\Q \cup \{ \infty \}$, is the $\rG$-orbit of cusp $\infty$.
 
 For any $\tau \in \h^{*}_{\rG}$, let $\bar{\rG_{\tau}}=\{ \g \in \rG | \g \cdot \tau =\tau \}$ be the stabilizer subgroup of $\tau$ in $\bar{\rG}$. For any $\cu \in \Cu_{\rG}$, $\bar{\rG_{\cu}}$ is an infinite order cyclic subgroup of $\bar{\rG}$. If $\cu =\infty$ then $\bar{\rG_{\infty}}$ is generated by $\overline{\left({1 \atop 0}{h_{{\infty}} \atop 1}\right)}$, for a unique real number $h_{{\infty}} > 0$, which is called the cusp width of the cusp $\infty$. Throughout the article, we shall denote $t_{\infty}=\left({1 \atop 0}{h_{{\infty}} \atop 1}\right).$ In case of $\cu \neq \infty$, $\bar{\rG_{\cu}}$ is generated by $\overline{A_{\cu} \left({1 \atop 0}{h_{{\cu}} \atop 1}\right)A_{\cu}^{-1}}$ for some smallest real number $h_{\cu} >0$, called the cusp width of the cusp $\cu$, where $A_{\cu}= \left({\cu \atop 1}{-1 \atop \ 0}\right) \in \mr{SL}_2(\R)$ so that $A_{\cu}(\infty)=\cu$. For rest of the article, we shall denote $t_{\cu}=A_{\cu} \left({1 \atop 0}{h_{{\cu}} \atop 1}\right)A_{\cu}^{-1}$. The reader may note that the stabilizer subgroup of any cusp $\cu$ in $\rG$ is not cyclic. This is essentially because, the stabilizer group $\rG_{\cu}$ contains both $t_{\cu}$ and $-t_{\cu}$, and in particular, $-I\in \rG_{\cu}$. This is not possible because $-I$ is a non-trivial element of finite order.

\subsection{Results}\label{sec:results}
\begin{thm}\label{thm:main} Let $\rG\subseteq \mathrm{SL}_2(\R)$ be any group, and $\rH$ be any finite index subgroup of $\rG$. Suppose that $-I \in \rH$, and $\bar{\rH},~\bar{\rG}$ are Fuchsian groups of the first kind. Let $\rho:\rH \to \mathrm{GL}_m(\mathbb{C})$ be any finite dimensional representation, and $\X(\tau)$ be a weakly holomorphic vvaf for $\rH$ associated to $\rho$. Then we have the following.
\begin{enumerate} 
\item[(a)] Let $\cu$ be an arbitrary cusp of $\rG$ and $\{\cu_i| 1\leq i\leq n_{\cu}\}$ be a set of inequivalent cusps of $\bar{\rH}$ for which $$\bar{\rG}\cdot \cu=\bigcup_{i=1}^{n_{\cu}} \bar{\rH}\cdot \cu_i.$$ If $\rho(t_i)$ is not diagonalizable for some $i,$ then $\widetilde{\rho}(t_{\cu})$ is not diagonalizable, where $t_i$ is a stabilizer of cusp $c_i$ such that $\overline{t_i}$ is the generator of the stabilizer subgroup of the cusp $\cu_i$ in $\bar{\rH}$.
\item[(b)] Let $\X(\tau)$ be a weakly holomorphic logarithmic vvaf associated with $\rho.$ Then the lifted form $\widetilde{\X}(\tau)$ is a weakly holomorphic logarithmic vvaf, associated with the induced representation $\widetilde{\rho}$ of the same weight as $\X(\tau).$
\end{enumerate} 
\end{thm}

The first author~\cite{Bajpai2019} showed that if all of the $\rho(t_{i})$ are diagonalizable, then $\widetilde{\rho}(t_{\cu})$ is also diagonalizable. The author achieved this by explicitly writing down the eigenbasis of $\widetilde{\rho}(t_{\cu})$, in terms of the eigenbasis of all the $\rho(t_{i})$s.

We show that the converse of~\cite[Theorem 4.2]{Bajpai2019} is true, in part $(a)$ of the theorem above. Our approach diverges, opting for a simpler method that circumvents the necessity of explicitly deriving the eigenbasis. Instead, we prove that a certain power of $\widetilde{\rho}(t_{\cu})$ adopts a block diagonal structure, with at least one block remaining nondiagonalizable.

Yet, using the methodology outlined in~\cite{Bajpai2019}, we carry on the task of writing down the eigenvalues of $\widetilde{\rho}(t_{\cu})$ and explicitly compute their multiplicities, in Lemma~\ref{cor:eigenvalues}. Although part $(a)$ of Theorem~\ref{thm:main} could theoretically be derived from Lemma~\ref{cor:eigenvalues}, we leverage this stronger outcome to facilitate the explicit derivation of the Fourier coefficients of the lift, as presented later in Theorem~\ref{thm:relation}.


For any even integer $k,$ it has been shown in~\cite{Bajpai2019} that there exists a bijection between the set of all admissible vvaf of weight $k$ associated to $\rho$ and the set of all admissible vvaf associated to $\widetilde{\rho}$ of weight $k.$ In this article, we shall define this isomorphism as a lift and work with this. Furthermore, we also have an analogous isomorphism for the logarithmic vector-valued automorphic forms, and we note this in Proposition~\ref{prop:isomorphism}.

As a consequence of the lifting result for admissible vector-valued automorphic forms, the first author showed the existence of weakly holomorphic admissible vector-valued automorphic forms. In the same spirit, we shall discuss the existence of weakly holomorphic logarithmic vector-valued automorphic forms in the form of Proposition~\ref{prop:existence} and Corollary~\ref{cor:existence}. Some technical issues are arising in this process, and we address them in Sections~\ref{sec:lift} and~\ref{sec:exist}.

In the second half of the article, we study the relations between the Fourier coefficients of $\X(\tau)$ and its lift $\widetilde{\X}(\tau).$ To get a Fourier expansion of $\X(\tau)$ at the cusps of $\bar{\rH},$ we need to assume that $\rho(\gamma)$ has only unitary eigenvalues for every parabolic element $\gamma$ in $\rH.$ In particular to have Fourier expansion of $\widetilde{\X}(\tau)$ at the cusps of $\rG,$ we need to ensure that $\widetilde{\rho}(\gamma)$ has only unitary eigenvalues for every parabolic element $\gamma$ of $\rG$. This requirement stems from the fact that the presence of non-unitary eigenvalues implies a non-polynomial growth of $\rho$, as discussed in ~\cite[Sect.~6]{BBF}. On the other hand, the necessity of such polynomial growth is demonstrated in~\cite[Thm.~1.2]{BBF}. Moreover, the first author explicitly described them in the case of admissible vector-valued automorphic forms. Here we do the same for logarithmic vector-valued automorphic forms, and as a consequence, we deduce the following.

\begin{thm}\label{thm:relation} Let $\rH, \rG$ and $\rho$ be as in Theorem~\ref{thm:main} with index $[\rG:\rH]=d$. Let $\X(\tau)$ be a weakly holomorphic vector-valued automorphic form associated to $(\rH,\rho)$ of weight $0$ and $\widetilde{\X}(\tau)$ be the lifted form associated to $(\rG,\widetilde{\rho}).$ Then there exists a family of weakly holomorphic vector-valued automorphic forms $\{\X^{(i)}(\tau)|1\leq i\leq d\}$ associated to some $\mathrm{SL}_2(\R)$ conjugates of $(\rH,\rho)$ with the following property: let $\X^{(i)}[n]$ be the $n^{th}$ Fourier coefficient of $\X^{(i)}(\tau),$ and $\widetilde{\X}[n]$ be the $n^{\mathrm{th}}$ Fourier coefficient associated to the lifted form $\widetilde{\X}(\tau).$ Then there exists a set of rationals $\{m_i,n_i,r_i \mid 1\leq i\leq d\}$, such that
$$\widetilde{\X}[n]=\left(r_i\X^{(i)}[m_in+n_i]\right)^t_{1\leq i\leq d}.$$
Furthermore, the quantities $\{m_i,n_i,r_i\mid 1\leq i\leq d\}$ depend only on $\rH,\rG$ and $\rho$, but not on the associated vvaf $\X(\tau)$.
\end{thm}

\subsection{Overview of the article} 
In Section~\ref{sec:fg}, we recall the basic notions and discuss some key facts from this vast area. In Section~\ref{sec:vvaf}, we discuss some properties of vector-valued automorphic forms in the settings of both admissible and logarithmic. These two sections are kind of expository and serve as preliminary for our work. In Section~\ref{sec:prop}, we recall the results of~\cite{Bajpai2019} along with the techniques involved, and in Section~\ref{sec:lift}, we use it to study the lift of logarithmic vector-valued automorphic forms. In Section~\ref{sec:exist}, we discuss their existence and some examples. In the end, we compute their Fourier coefficients in Section~\ref{sec:computation}.

\subsection{Notations} We denote $\text{SL}_2(\R)$ to be the group of all $2\times 2$ real matrices with determinant one, and $\text{PSL}_2(\R)$ to be $\text{SL}_2(\R)/\{\pm I\}.$ For any integer $m\geq 1,~I_m$ is denoted to be the identity matrix in $\text{GL}_m(\R)$. We shall write $\tau = x + \Ii y$ a point in $\mathbb{H}$, and $\text{im}(\tau)$ be the corresponding imaginary part. For any unitary $\lambda,$ we always denote $\mu(\lambda)$ to be the unique real number such that $\lambda=\exp(2\pi \Ii \mu(\lambda))$ and $0\le\mu(\lambda)<1$. Unless otherwise specified, every constant appearing in the draft depends on the Fuchsian group. We write $f \ll g$ for $|f|\leq c|g| $ where $c$ is a constant irrespective of the domains of $f$ and $g$. Given a vector $v:=(v_1,v_2,\cdots, v_m)$ in $\C^m,$ we shall always denote $v^t$ to be its transpose, that is the column vector $\begin{psmallmatrix} v_1\\ v_2\vspace{-0.24cm}\\ \vdots \\ v_m \end{psmallmatrix}.$

\section{Vector-valued automorphic forms}\label{sec:vvaf}

This section reviews the basics of vector-valued automorphic forms that we need to understand and prove the main results of the article. In~\cite{Bajpai2019} and~\cite{BBF}, the authors studied vector-valued automorphic forms associated with Fuchsian groups of the first kind of even weight.

 Given an element $\gamma=\tmt{a}{b}{c}{d}\in \rG,$ we denote $j(\gamma,\tau)=c\tau+d.$ However, if we take an element $\gamma=\overline{\tmt{a}{b}{c}{d}}\in \text{PSL}_2(\R),$ we can not define 
 $j(\gamma,\tau)$ in this way because of one possible sign factor. In~\cite{Bajpai2019} and~\cite{BBF}, this sign factor was redundant as the authors worked with even weights vector-valued automorphic forms, and therefore, it always comes with even power. Therefore, the even weight is necessary to define $j(\cdot,\cdot)$. In this article, we shall mainly work inside $\text{SL}_2(\R)$. This will allow us a slight improvement, i.e. we can work with any integer weight, not just the even ones. Our treatment of vector-valued automorphic forms in this section closely follows~\cite {Gannon1}.

\begin{defn}[Stroke operator] \label{stroke_operator} \rm If $\X:\h\to\C^m$ is a vector-valued meromorphic function, $\gamma\in \mr{SL}_2(\R)$ and $k$ be a complex number, we define a vector-valued meromorphic function $\X|_k\gamma$ on $\h$ by setting $\X|_k\gamma(\tau)=j(\gamma,\tau)^{-k}\X(\gamma\tau),$
where $j(\gamma,\tau)=c\tau+d$ when $\gamma=\tmt{a}{b}{c}{d}.$
\end{defn}
It is needless to mention that the function $j(\gamma,\tau)^{k}$ is well defined because $j(\gamma,\cdot):~\mathbb{H}\to \mathbb{C}$ is a non-vanishing function. One can show that the Stroke operator induces a right group action of $\text{SL}_2(\R)$ on the space of vector-valued meromorphic functions on $\h$. For this article, we shall restrict to integer $k$. We are interested in the vector-valued meromorphic functions with finite order poles at the cusps. More precisely, let us first make the following definition.

\begin{defn} \rm Let $\X:\h\to\C^m$ be a vector-valued meromorphic function. 
\begin{itemize}

\item  We say that $\X(\tau)$ has \emph{moderate growth at $\infty$} when there exists $c \in \R$ and $Y>0$ such that $\|\X(\tau)\|< \exp(c y)$ when $y>Y$, where $y=\mr{Im}\tau.$
\item we say that $\X(\tau)$ has \emph{moderate growth at $\cu \in \R$} with respect to $k\in \mathbb Z$ when $\X|_kA_\cu$ has moderate growth at $\infty$.
\end{itemize}
\end{defn}
Furthermore, we shall refer to the composition of $X(\tau)$ with the projection map to each coordinate of $\mathbb{C}^m$, as the components of $\X(\tau)$. To introduce vector-valued automorphic forms,  we first discuss the representations of $\rG$. For us, these representations are divided into two types. Namely, admissible and logarithmic (non-admissible) ones. Let us describe them in the following form. 
\begin{defn} \rm
Consider a representation $\rho:\rG \to \mathrm{GL}_m(\mathbb{C})$ such that for every parabolic element $\gamma \in \rG$, the eigenvalues of $\rho(\gamma)$ are all unitary. We say that $\rho$ is an {\it admissible representation} of $\rG$ if $\rho(\gamma)$ is diagonalizable for every parabolic element $\gamma\in \rG
$. Otherwise, we say that $\rho$ is a \emph{logarithmic representation.}
\end{defn}
Any parabolic element in $\bar{\rG}$ is power of $\overline{t_{\cu}}$, for some cusp $\cu\in \Cu_{\rG}.$ In particular, any parabolic element in $\rG$ is power of $\pm t_{\cu}$. Note that $\rho(-I)$ is always diagonalizable. Therefore, we are really talking about diagonalizability and the unitary condition at the finitely many parabolic elements $\{t_{\cu}\}_{\cu\in \Cu_{\rG}}$. We shall now separately treat the vector-valued meromorphic functions associated with them.

\begin{defn}\rm 
Let $\bar{\rG}$ be a Fuchsian group of the first kind, $k$ be an integer, $\rho:\rG\to \mr{GL}_m(\C)$ be an admissible representation, and $\X:\h\to\C^m$ be a vector-valued meromorphic function. Then,  we say that $\X(\tau)$ is a weakly holomorphic \emph{admissible vvaf} of weight $k$ with respect to $\rho$ if it satisfies the following properties.
\begin{itemize}
\item It is a holomorphic function on $\h,$ with moderate growth at any cusp of $\rG,$
\item It has finitely many poles in the closure of a fundamental domain of $G'$, and
\item $\X|_k\g(\tau) = \rho(\g) \X(\tau),~\forall \g \in \rG.$\label{functional}
\end{itemize}
\end{defn}
\begin{remark}\rm $\;$
\begin{enumerate}
\item[(a)] If there exists a non-zero $\X(\tau)$ satisfying $\X|_k\g(\tau) = \rho(\g) \X(\tau),~\forall \g \in \rG$, then we must have that $\rho(I)=e^{-i\pi k}\rho(-I).$
\item [(b)]Of course, a representation $\rho$ satisfies $\rho(-I)=e^{i\pi k}\rho(I),$ only if $k$ is an integer. Moreover, $\rho$ can be considered a representation of $\bar{\rG}$ if and only if $\rho(-I)=\rho(I).$ Equivalently when the associated vvaf has even weight. In other words, only a vvaf of even weight is associated with $\rG,$ is a vvaf of even weight associated with $\bar{\rG}$. Therefore we are studying the problem in a bit more generality, even in the case of admissible forms.
\item[(c)] If $\rho$ is any irreducible representation, it turns out that $\rho(-I)$ must be one of the $\pm I.$ This is a simple fact from linear algebra. 
\end{enumerate}
\end{remark}

Recall that $\overline{t_{\cu}}\in \mr{PSL}_2(\R)$ (up-to conjugation) generates $\bar{\rG_{\cu}}$. Since $\rho$ is admissible, we can write $\rho(t_{\cu})=P_{\cu} \mathrm{diag}\left(\lambda_{1,\cu},\lambda_{2,\cu},\cdots,\lambda_{m,\cu}\right)P_{\cu}^{-1}$ for any cusp $\cu$ of $\rG.$ To get the Fourier expansion around the cusps, we denote 
$$ \tq_{_\cu}=\left\{\begin{array}{cccccccc}  
e^{\frac{2\pi \Ii A_{_{\cu}}^{-1}\tau}{h_{\cu}}} & \mr{if}~\cu \in \Cu_{_{\rG}},\cu \neq \infty\\  
e^{\frac{2\pi\Ii \tau}{h_{\infty}}} & \mr{if }~\cu=\infty. 
\end{array}\,. \right .$$

Following \cite[Proposition 3.11]{BBF}, at the cusp $\infty$ we have the following Fourier expansion,
\begin{equation*}\label{fourier_series_vvaf}
   \X(\tau) = P_{\infty} \tq^{\Lambda} P_{\infty}^{-1}\sum_{n=-M_{\infty}}^\infty \X_{[n]} \tq^{n}.
   \end{equation*}
   For any other cusp $\cu\neq \infty,$ consider $\Y(\tau)=\X|_kA_{\cu}(\tau),$ and then $$\X(\tau)=\Y|_kA_{\cu}^{-1}(\tau)=(\cu-\tau)^{-k}\X_0(A_{\cu}^{-1}\tau).$$
   Following this, we have the Fourier expansion
   \begin{equation}\label{eqn:admf}
\X(\tau) = (\tau - \cu)^{-k} P_{\cu}\tq_{\cu}^{\ \Lambda_{\cu}}P_{\cu}^{-1} \sum_{n=-M_{\cu}}^{\infty}  \X^{^{\cu}}_{[n]} \ \tq_{_{\cu}}^{\ n}, \ \X^{^{\cu}}_{[n]} \in \C^{d},
\end{equation}
where we denote $\tq_{\cu}^{\ \Lambda_{\cu}}$ to be the diagonal matrix $(\tq^{\mu(\lambda_{1,\cu})},\tq^{\mu(\lambda_{2,\cu})},\cdots, \tq^{\mu(\lambda_{m,\cu})}).$ In particular, note that the Fourier coefficients do not depend on the choice of the diagonalizing matrix. Note that $\X(\tau)$ is a weakly holomorphic admissible vvaf if $M_{\cu}$ is always finite.

A weakly holomorphic admissible vvaf $\X(\tau)$ of weight $k$ is called \emph{holomorphic} if it has no poles in $\h$ and, for any cusp $\cu$ of $\rG$, the function $\X|_kA_\cu(\tau)$ is bounded in some half-plane (contained in $\h$), that is to say simply that  $\X(\tau)$ is holomorphic everywhere in $\h^{*}_{\rG}$. It is called a weight $k$ \emph{cusp form} if, for any cusp $\cu$, the function $\X|_kA_\cu(\tau)$ approaches to $0$ as $y\to\infty$. This is same as saying that $\X(\tau)$ is a holomorphic admissible vvaf if $M_{\cu}\leq 0$, and an admissible cusp form if $M_{\cu}<0$ for any cusp $\cu$ of $\rG.$

We are interested in the representation $\rho$ of $\rG$, which is not admissible. In other words, $\rho(\gamma)$ is not diagonalizable for some parabolic element $\gamma\in \rG$. We call such a representation logarithmic. In this case we can write $\rho(t_{\cu})$ in the Jordan canonical form as
\[ P_{\cu}\begin{psmallmatrix}
   J_{m(\lambda_{1,\cu}),\lambda_{1,\cu}} & & \\
    & J_{m(\lambda_{2,\cu}),\lambda_{2,\cu}} & &\\
    & & \ddots &\\ 
    & & & J_{m(\lambda_{k,\cu}),\lambda_{k,\cu}}\\ 
  \end{psmallmatrix}P_{\cu}^{-1},\]
for any cusp $\cu$ of $\rG,$ where the Jordan block $J_{m,\lambda}$ is defined to be 
{\small $$\begin{psmallmatrix}
   \lambda & & &  &\\
    \lambda &  \ddots&  & &\\ 
     & \ddots &  & \ddots &\\
     &   & \lambda &  &\lambda\\ 
 \end{psmallmatrix}_{m\times m},$$ }
which is conjugate to the canonical Jordan block. We also set $R(t_{\cu})$ to be the collection of the eigenvalues of $\rho(t_{\cu}),$ which may not necessarily be pairwise distinct.

Let $\X:\h\to\C^m$ be a vector-valued meromorphic function, and $\rho:\rG \to \mathrm{GL}_m(\mathbb{C})$ be a logarithmic representation. Suppose that for some integer $k,~\X(\tau)$ satisfies
\begin{itemize}
\item $\X|_k\g = \rho(\g) \X,~\forall \g \in \rG,$
\item $\rho(I)=e^{i\pi k} \rho(-I).$
\end{itemize}
Now let $\cu$ be any arbitrary cusp of $\rG.$ Then for each eigenvalue $\lambda$ of $\rho(t_{\cu})$ there are $\tq_{\cu}$-expansions \begin{equation}\label{eq:hjcm}
h_{\lambda,j,\cu}(\tau)=P_{\cu}\tq_{\cu}^{\mu(\lambda)}P_{\cu}^{-1}\sum_{n=-M_{\cu}}^{\infty} \X^{\cu}[\lambda,j,n]\tq_{\cu}^{n},~0\leq j\le m(\lambda)-1
\end{equation}
such that, at the cusp $\infty$
\begin{equation} \label{eqn:logarithmic}
\X(\tau)=\sum_{\lambda \in \Sp}\sum_{j=0}^{m(\lambda)-1}(\log \tq)^{j}h_{\lambda,j,\infty}(\tau),
\end{equation}
and at the cusp $\cu (\neq \infty)$,   
\begin{equation}\label{eqn:finite}
\X(\tau) = (\tau - \cu)^{-k} \sum_{\lambda \in \Sp}\sum_{j=0}^{m(\lambda)-1}(\log \tq_{\cu})^{j}h_{\lambda,j,\cu}(\tau),
\end{equation}
where $m(\lambda)$ is the size of one of the Jordan blocks associated with the eigenvalue $\lambda.$
\begin{defn} Let $\X:\h \to \mathbb{C}^m$ be a vector-valued holomorphic function and suppose there exists an integer $k$ such that $\X(\tau)$ satisfies~(\ref{eqn:logarithmic}) and~(\ref{eqn:finite}) above with respect to a logarithmic representation $\rho: \rG\to \mr{GL}_m(\C)$. Then, we say that $\X(\tau)$ is a weakly holomorphic logarithmic vvaf of weight $k$, if for any cusp $\cu$ of $\rG$, all of the $\tq_{\cu}$ expansions $h_{\lambda, j,\cu}(\tau)$ start from some finite $n.$
\end{defn}
In particular, we say that $\X(\tau)$ is holomorphic at the cusp $\infty$ when $M\leq 0$, and holomorphic at the cusp $\cu$ when $\X|_kA_\cu(\tau)$ is holomorphic at the cusp $\infty$, or equivalently when $M_{\cu}\leq 0.$ If $\X(\tau)$ is holomorphic at all cusps, then we say that $\X(\tau)$ is a holomorphic logarithmic vvaf. In addition, we say that $\X(\tau)$ is a logarithmic vector-valued cusp form if $M_{\cu}<0$ for any cusp $\cu$ of $\rG.$


\section{Properties of lifted vector-valued automorphic forms}\label{sec:prop} In this section, we shall recall the induction of representations and introduce vvaf associated with them. We closely follow~\cite{Bajpai2019} throughout this section and develop the main tools to prove the main results of this article. 

\subsection{A special choice of the representatives}\label{sec:desc}
Let us first recall the usual setup: $\rG\subseteq \mathrm{SL}_2(\R)$ be any group, and $\rH$ be any finite index subgroup of $\rG$. Suppose that the index is $d,~-I \in \rH$, and $\bar{\rH},~\bar{\rG}$ are Fuchsian groups of the first kind. Fix any cusp $\cu \in \widehat\Cu_{_{\rG}}$ and let $\cu_{1}, \cdots, \cu_{\n_{_{\cu}}}$ be the representatives of the $\bar{\mr{H}}$-inequivalent cusps which are $\bar{\rG}$-equivalent to the cusp $\cu$, so we have $ \bar{\rG}\cdot \cu=\bigcup_{i=1}^{\n_{_{\cu}}} \bar{\mr{H}}\cdot{\cu_{i}}.$ In particular, one can write 
$$ \rG\cdot \cu=\bigcup_{i=1}^{\n_{_{\cu}}} \mr{H}\cdot{\cu_{i}},$$
since two cusps are $\bar{\rH}$ equivalent, if and only if, they are $\rH$ equivalent. Therefore, for each $i$ we get an $A_i\in \rG$ with $A_{i}(\cu)=\cu_{i}$. Let us denote $k_{_{\cu}}$ be the cusp width of $\cu$ in $\rG$ and $h_{\cu_{i}}$ be the cusp width of $\cu_{i}$ in $\mr{H}$. Then a set of coset representatives of $\mr{H}$ in $\rG$ can be taken to be $g_{ij}=t_{\cu}^{j} A_{i}^{^{-1}}$ for all $1\leq i\leq n_{\cu}$ and $0\leq j < h_{i}$, where $h_{i}=\frac{h_{\cu_{i}}}{k_{_{\cu}}} \in \Z.$ It turns out that, $\sum_{1\leq i\leq n_{\infty}}h_i=d$. Setting $t_i=A_it_{\cu}^{h_i}A_i^{-1}$, note that any element of the stabilizer group $\rH_{\cu_i}$ can be written as power of $\pm t_i.$ 

Let $\rho$ be a representation of rank $m$ associated to $\rH,$ and denote $\widetilde{\rho}:=\text{Ind}_{\rH}^{\rG}(\rho)$ to be the induction of $\rho.$ With the choice of coset representatives $\{g_{i,j}\}$ of $\rH$ in $\rG$ as described above, we can write $\widetilde{\rho}(t_\cu)$ in the block diagonal form, where each block is of size $mh_i\times mh_i,~\forall~1\leq i\leq n_{\cu}.$ Moreover, these blocks are in the lower-diagonal form whose right top block is $\rho(t_{i}),$ and all other blocks are in the lower diagonal entry is $I_{m\times m}.$ More precisely, it is of the form
 {\small $$\begin{psmallmatrix}
    0 & & &  & \rho(t_i)\\
    I &  \ddots&  & &\\ 
     & \ddots &  & \ddots &\\
     &   & I &  &0\\ 
 \end{psmallmatrix}_{mh_i\times mh_i}.$$ }
For a more detailed discussion of Fuchsian groups, we refer the reader to~\cite[Section 4]{Bajpai2019}. Here we are pointing out the similar facts over their pullbacks in $\text{SL}_2(\R).$

\subsection{Lifting of vector-valued automorphic forms}\label{sec:liftinduction} 
Let $\rH,\rG$ and $\rho$ be as in the previous section, and $\X(\tau)$ be a weakly holomorphic vvaf associated to $(\rho,\rH)$. Now one may ask for a lifted form $\widetilde{X}(\tau),$ which is also a weakly holomorphic vvaf associated with $\widetilde{\rho}.$ Fix a cusp $\cu$ of $\rG$ and let $\{g_{i,j}\}$ be the set of coset representatives of $\rH$ in $\rG,$ as described earlier. We then define the induced function $\widetilde{\X}^{(\cu)}:\h \to \C^{dm}$ by setting,
$$\tau \mapsto \bigg( \X(g_{i,j}^{-1}\tau)\bigg)^{\mathfrak{t}}_{\substack{1\leq i\leq n_{\cu}\\0\leq j<h_i}}.$$
The reader can note that, given any cusp $\cu$ of $\rG,$ we are uniquely lifting the vvaf $\X(\tau)$ to $\widetilde{\X}^{(\cu)}(\tau),$ because $g_{i,j}$ are well defined. Of course, $\widetilde{\X}^{(\cu)}(\tau)$ is just a vector-valued holomorphic function on $\h$ right now, because $\X(\tau)$ is a weakly holomorphic vvaf. To make sure that $\widetilde{\X}^{(\cu)}(\tau)$ is a weakly holomorphic vvaf (be it admissible or logarithmic), we first need to check the functional equation with respect to the induced representation $\widetilde{\rho}$. Note that,
\begin{align*} \footnotesize
\widetilde{\X}^{(\cu)}(\g\tau) = \left(\begin{array}{ccc}
\X(\g_{1}^{-1}\g\tau)\\
\X(\g_{2}^{-1}\g\tau)\\ 
\vdots\\ 
\X(\g_{d}^{-1}\g\tau)\\
\end{array} \right) &= \left( \begin{array}{ccc}
\X(\g_{1}^{-1}\g\g_{j_{1}}\g_{j_{1}}^{-1}\tau)\\
\X(\g_{2}^{-1}\g \g_{j_{2}}\g_{j_{2}}^{-1}\tau) \\
\vdots\\
\X(\g_{d}^{-1}\g \g_{j_{d}}\g_{j_{d}}^{-1}\tau)\\
\end{array} \right) \\  
&=\left( \begin{array}{ccc}
j(\g_{1}^{-1}\g \g_{j_{1}},\tau)^k\rho(\g_{1}^{-1}\g \g_{j_{1}}) \X(\g_{j_{1}}^{-1}\tau) \\
j(\g_{2}^{-1}\g \g_{j_{2}},\tau)^k\rho(\g_{2}^{-1}\g \g_{j_{2}}) \X(\g_{j_{2}}^{-1}\tau) \\
\vdots\\
j(\g_{d}^{-1}\g \g_{j_{d}},\tau)^k\rho(\g_{d}^{-1}\g \g_{j_{d}}) \X(\g_{j_{d}}^{-1}\tau) \\
\end{array} \right) \,,\end{align*}
where $\{\g_i| 1\leq i\leq d\}$ is the set $\left\{g_{i,j}| 1\leq i\leq n_{\cu}, 0\leq j<h_i\right\}.$ To satisfy the functional equation property, we need that
$$j(\g_{i}^{-1}\g \g_{j_{i}},\tau)^k=j(\g,\tau)^k,~\forall~ 1\leq i\leq d,\gamma \in \rG.$$
We can make sure this if, $r_2(\g_{i}^{-1}\g \g_{j_{i}})=r_2(\g),~\forall~1\leq i\leq d,~\gamma\in \rG, $ or simply if the weight $k=0,$ where $r_2(\cdot)$ denotes the second row of the corresponding matrix. Of course, it is unlikely that $r_2(\g_{i}^{-1}\g \g_{j_{i}})=r_2(\g),~\forall~1\leq i\leq d,~\gamma\in \rG $ would always hold. So to be on the safer side, we stick to the weight $k=0$ case to make sure that the functional equation is satisfied. Before discussing the moderate growth condition, let us first define a suitable lift for any arbitrary weight case. For this, let us first recall a reduction trick introduced in~\cite{Bajpai2019}. The idea is to find a scalar-valued cusp form $\Delta_{\rG}(\tau)$ of non-zero weight, which is holomorphic on $\mathbb{H}^{*}_{\rG}$ and nonzero everywhere, except at $\infty.$ For instance if $\rG$ is given by the modular group $\text{SL}_2(\mathbb{Z}),$ then one can take 
$$\Delta_{\rG}(\tau)=(\eta(\tau))^{24}=q\prod_{n\geq 1}(1-q^n)^{24},$$
where $\eta(\tau)$ is the Dedekind eta function. For the existence in the general case, the reader may look at the exposition in~\cite[Section 4]{Bajpai2019}. 

Now given a weakly holomorphic admissible or logarithmic vvaf $\X(\tau)$ of weight $k,$ let us denote $\X_0(\tau)=\Delta_{\rH}^{-k/w_{\rH}}(\tau)\X(\tau),$ where $w_{\rH}$ is the weight of $\Delta_{\rH}(\tau).$ It is clear that, $\X_{0}(\tau)$ is a weakly holomorphic vvaf of weight $0,$ associated to the representation $\rho \otimes \nu_{\rH}^{-k},$ where $\nu_{\rH}$ is the rank $1$ representation associated to $\Delta_{\rH}^{1/w_{\rH}}(\tau).$ One can then consider a reduction to a weight $0$ automorphic form by $\X(\tau) \mapsto \X_0(\tau).$ In particular, we now have a recipe to lift to a vvaf of arbitrary weight by considering the map 
$$\X(\tau) \mapsto \X_0(\tau) \mapsto \widetilde{\X_0}^{(\cu)}(\tau)\Delta_{\rG}^{k/w_{\rG}}(\tau):= \Delta_{\rG}^{k/w_{\rG}}(\tau)\bigg( \X_0(g_{i,j}^{-1}\tau)\bigg)^{\mathfrak{t}}_{\substack{1\leq i\leq n_{\cu}\\0\leq j<h_i}}.$$ Note that $\widetilde{\X_0}^{(\cu)}(\tau)\Delta_{\rG}^{k/w_{\rG}}(\tau)$ satisfies the functional equation with respect to the representation $(\widetilde{\rho\otimes \nu_{\rH}})^{-k}\otimes \nu_{\rG}^{k}=\widetilde{\rho}.$ Therefore, we refine our definition of lift by setting 
$$\widetilde{\X}^{(\cu)}(\tau):=\widetilde{\X_0}^{(\cu)}(\tau)\Delta_{\rG}^{k/w_{\rG}}(\tau).$$
If $\infty$ is a cusp of $\rG,$ then we set $\widetilde{\X}(\tau):=\widetilde{\X}^{(\infty)}(\tau)$ as the definition of lifted form. If not, we pick any cusp $\cu$ of $\rG$ and set $\widetilde{\X}(\tau):=\widetilde{\X}^{(\cu)}|_kA_{\cu}(\tau),$ which satisfies the required functional equation with respect to the representation $A_{\cu}^{-1}\gamma A_{\cu}\mapsto \widetilde{\rho}(\gamma).$

\subsection{Preservation of the cuspidal properties} Before studying the behavior at the cusps, we first need to have a better understanding of the representatives of $\rH$ in $\rG.$ Suppose that $\infty$ is a cusp of $\rG$ and $\left\{\cu_i|1\leq i\leq n_{\infty}\right\}$ are the cusps of $\rH$ lying under $\infty,$ and $\{g_{i,j}\}$ be the set of coset representatives of $\rH$ in $\rG$ as described before, with the important property that, $g_{i,j}(\cu_i)=\infty.$ The following lemma is about a comparison with the set of all coset representatives $\left\{g_{i,j}| 1\leq i\leq n_{\infty}, 0\leq j<h_i\right\}$ inside $\rG$ and the set $\left\{A_{\cu_i}| 1\leq i\leq n_{\infty}\right\}$ inside $\text{SL}_2(\R).$

\begin{lem}\label{lem:gtu} There exists $a_{i,j}$ and $\alpha_{i,j}\in \mathbb{R}$ such that
$$ g_{i,j}\tau=a_{i,j}^2A_{\cu_i}^{-1}\tau+jh_{\infty}+\alpha_{i,j},~\forall~\tau \in\mathbb{H}.$$
\end{lem}
\begin{proof} Recall that, $g_{i,j} = t_{\infty}^jA_i^{-1}$ where $A_i(\infty)=\cu_i.$ We also know that $A_{\cu_i}(\infty)=\cu_i,$ in particular, $A_i^{-1}A_{\cu_i}(\infty)=\infty.$ On the other hand, $A_i^{-1}A_{\cu_i}\in \text{SL}_2(\mathbb{R}).$ In particular, $A_i^{-1}A_{\cu_i}=\tmt{a}{\alpha}{0}{1/a}$ for some $a,\alpha \in \mathbb{R}.$ We then have,
\begin{equation*}\label{eqn:gtau}
g_{i,j}\tau=t_{\infty}^{j}A_i^{-1}\tau=t_{\infty}^{j}\tmt{a}{\alpha}{0}{1/a}A_{\cu_i}^{-1}\tau=a^2A_{\cu_i}^{-1}\tau+jh_{\infty}+\alpha a.
\end{equation*}
This completes the proof by taking $a_{i,j}:=a $ and $\alpha_{i,j}:=\alpha a.$
\end{proof}

We recall that any classical holomorphic modular form does not have weight $0,$ unless it is a constant function. In those cases, the representation under consideration has a finite image. However, the same may not be true if the representation does not have a finite image. Consider the representation $\mathbb{I}_0:\text{SL}_2(\Z)\to \text{GL}_2(\C),$ given by $\gamma \mapsto \gamma,$ and the holomorphic function $\Y:\mathbb{H}\to \C^2$ given by $\tau \mapsto (\tau,1)$. Note that $\Y(\tau)$ is a holomorphic logarithmic vvaf of weight $-1$ associated with $\mathbb{I}_0$. In particular, $\Y'(\tau):=\Y(\tau)\Delta(\tau)^{\frac{1}{12}}$ is a non-constant holomorphic logarithmic vvaf of weight $0$ associated to the representation $\rho':=\mathbb{I}_0 \otimes \nu_{\text{SL}_2(\Z)}.$ In fact, following~\cite{Gannon1}, we know that the space of holomorphic logarithmic vvaf is a free module of rank two over the polynomial ring $\C[E_4,E_6]$, generated by $\Y'(\tau)$ and its modular derivative.

Moreover, the representation $\rho':\text{SL}_2(\Z)\to \text{GL}_2(\C)$ indeed have infinite image, because 
$$\norm{\rho'(t^n)}=\norm{\mathbb{I}_0 \otimes \nu_{\text{SL}_2(\Z)}(t^n)}=\norm{\mathbb{I}_0(t^n)}=n,$$
which follows from the explicit description of $\nu_{\text{SL}_2(\Z)}$, see~\cite[Page~7]{Bajpai2019}. Gannon in~\cite{Gannon1} gave an explicit description of the logarithmic representations of rank $2$, associated to $\text{SL}_2(\Z)$, and they all differ from $\rho'$ by some character of $\text{SL}_2(\Z).$ 

We are now ready to prove the required cuspidal properties of the lifted forms.
 \begin{lem}\label{lem:cuspidal} Let $\cu$ be an arbitrary cusp of $\rG$ and $\left\{\cu_i|1\leq i\leq n_{\cu}\right\}$ be the set of all inequivalent cusps of $\rH$ lying under $\cu.$ Then we have the following.
\begin{enumerate}
\item[(i)] If $\X(\tau)$ has moderate growth at all the cusps $\left\{\cu_i|1\leq i\leq n_{\cu}\right\},$ then the lifted form $\widetilde{\X}^{(\cu)}(\tau)$ has moderate growth at $\cu$ as well.
\item[(ii)] If $\X(\tau)$ is holomorphic (or vanishes) at all the cusps $\left\{\cu_i| 1\leq i\leq n_{\cu}\right\},$ then the lifted form $\widetilde{\X}^{(\cu)}(\tau)$ has same properties at the cusp $\cu,$ provided that the weight of $\X(\tau)$ is $0.$
\end{enumerate}
\end{lem}

\begin{proof} For both of the parts, it is enough to prove that $\widetilde{\X}^{(\cu)}(\tau)$ satisfy the required cuspidal properties at the cusp $\infty.$ 

Let us first prove part $(i)$. It is clear that $\X_0(\tau)$ has moderate growth at all the cusps $\left\{\cu_i| 1\leq i\leq n_{\infty}\right\},$ because any power of $\Delta_{\rH}(\tau)$ has the same property. Now we shall show that all the components of $\widetilde{\X}_{0}(\tau)$ have moderate growth at $\infty.$ Let $\Y(\tau):=\X_0(g_{i,j}^{-1}\tau)$ be such a component. Since $\X_0(\tau)$ has moderate growth at all the cusps $\left\{\cu_i| 1\leq i\leq n_{\infty}\right\}$ there exists a constant $c\in \mathbb{R}$ such that $\norm{\X_0(\tau)}\ll |e^{2\pi ic A_{\cu_i}^{-1}\tau}|$ as $\text{im}(\tau)\to \infty.$ It follows from Lemma~\ref{lem:gtu} that, $\norm{\Y(\tau)}\ll |e^{2\pi ic'\tau}|,$ for some constant $c'\in \R,$ as $\text{im}(\tau)\to \infty.$ This shows that, $\widetilde{\X}_{0}(\tau)$ has moderate growth at $\infty.$ On the other hand, any power of $\Delta_{\rG}(\tau)$ has moderate growth at $\infty$ as well, and this completes the proof of part $(i).$ 

Let us now prove part $(ii)$. To show that $\widetilde{\X}(\tau)$ is holomorphic (or vanishes) at the cusp $\infty,$ we need to show that all the components are bounded (or vanishes) as $\text{im}(\tau)\to \infty.$ The constant $c'$ appearing in the previous paragraph is $0$ when $\X(\tau)$ is holomorphic and negative when $\X(\tau)$ is a cuspform. Moreover, it can also be seen from Lemma~\ref{lem:gtu} that the constants $c$ and $c'$ from the previous paragraph are a positive multiple of each other. Therefore $\widetilde{\X}(\tau)$ has the similar cuspidal properties as $\X(\tau).$ This shows that the lifted form also shares the same cuspidal properties when the weight is $0.$
\end{proof}
\section{Lifting of logarithmic vector-valued automorphic forms}\label{sec:lift}
It was proved, by the first author in~\cite{Bajpai2019}, that the induction of an admissible representation is admissible. In this section, we shall study the induction of non-admissible, i.e., logarithmic representations. In this regard, we have the following.
\begin{prop}\label{prop:logarithmic} Let $\cu$ be an arbitrary cusp of $\rG$ and $\{\cu_i| 1\leq i\leq n_{\cu}\}$ be the set of inequivalent cusps of $\rH$ for which $\rG\cdot \cu=\bigcup H\cdot \cu_i.$ Then,
\begin{enumerate}
\item[(i)] If $\rho(t_{i})$ is not diagonalizable for some $i,$ then $\widetilde{\rho}(t_{\cu})$ is not diagonalizable.
\item[(ii)] In particular if $\rho$ is a logarithmic representation, then $\widetilde{\rho}$ is a logarithmic representation as well. 
\end{enumerate}
\end{prop}

\begin{proof} Let us start with considering the coset representatives $\{g_{i,j}\}$ of $\rH$ in $\rG$ from Section~\ref{sec:desc}. We first claim that, for any pair $(i,j)$, some non-trivial power of $\gamma_{i,j}=g_{i,j}^{-1}t_{\cu}g_{i,j}$ is in $\rH$. To prove this, we start by noting that there exists $n_{i,j,1}$ and $n_{i,j,2}$ such that $g_{i,j}^{n_{i,j,1}}\rH=g_{i,j}^{n_{i,j,2}}\rH$. This is because $\rH$ has a finite index in $\rG.$ In particular, for each pair $(i,j),$ there exists some integer $n_{i,j}$ such that $g_{i,j}^{-1}t_{\cu}^{n_{i,j}}g_{i,j}\in \rH$. Let us now consider $n=\text{lcm}\left\{n_{i,j}| 1\leq i\leq n_{\cu}, 1\leq j\leq h_i \right\}$. In particular, $g_{i,j}^{-1}t_{\cu}^{n}g_{i,j}\in \rH$ for each pair $(i,j).$ Therefore, $\widetilde{\rho}(t_{\cu}^n)$ is a block diagonal matrix where each block is of the form $\rho(g_{i,j}^{-1}t_{\cu}^ng_{i,j})$. Note that $g_{i,j}^{-1}t_{\cu}^ng_{i,j}\in \rH$ and fixes the cusp $\cu_i,$ hence it is some non-trivial power of $t_i.$ Let us write $g_{i,j}^{-1}t_{\cu}^ng_{i,j}=t_i^m,$ where $m\neq 0$ is an integer. By the assumption, there exists some $i$ for which $\rho(t_{i})$ can be written in Jordan normal form, with a Jordan block, say $J_{\lambda}$, of size greater than $1.$ In particular, $\rho(t_{i}^{m})$ can be written in a block diagonal form, where one of the blocks is $J_{\lambda}^m,$ which is not diagonalizable for any integer $m\neq 0$. This completes the proof of part $(i)$.

For the proof of part $(ii)$, let $\cu_i$ be a cusp of $\rH$ for which $\rho(t_i)$ is not diagonalizable. Now $\cu_i$ is a cusp of $\rG$, with $\cu_i$ itself lying under it as a cusp of $\rH.$ It the follows from the part $(i)$ that $\widetilde{\rho}(t_{\cu})$ is not diagonalizable, which completes the proof.
\end{proof}

\begin{remark}\rm
It is clear that $n_{i,j}\leq d,$ for each pair $(i,j),$ where $d$ is the index of $\rH$ in $\rG.$ In other words, $n$ is crudely bounded by $d^d.$ However, if $\rH$ is normal in $\rG,$ one can always take $n_{i,j}$ to be $d.$ In particular, one can take $n=d.$ In fact, the discussion in Section~\ref{sec:desc} allows us to take $n=\text{lcm}\left\{h_i | 1\leq i\leq n_{\cu}\right\}.$
\end{remark} 
\subsection{Eigenvalues of the induced representation} Let us first recall the usual setup. Fix a cusp $\cu$ in $\rG,$ and consider the decomposition
$$\rG\cdot \cu=\bigcup_{i=1}^{n_{\cu}}\rH\cdot \cu_i.$$
If we know all the eigenvalues of $\rho(t_i)$, then what do we know about the eigenvalues of $\widetilde{\rho}(t_{\cu})?$ The first author studied this when $\rho$ is an admissible representation and explicitly described the corresponding eigenspaces. Of course, the question arises if $\rho$ is not admissible. In other words, if one of the $\rho(t_{i})$ is not diagonalizable. We shall answer this in the following lemma. 
\begin{lem}\label{cor:eigenvalues} Let $\cu$ be any arbitrary cusp of $\rG,$ and $\left\{\cu_i| 1\leq i\leq n_{\cu} \right\}$ be the set of inequivalent cusps of $\rH$ lying under $\cu.$ Let $\left\{\lambda_{(i,k)} | 1\leq k\leq e_i \right\}$ be the set of distinct eigenvalues of $\rho(t_i)$ for each $1\leq i\leq n_{\cu}.$ Then, the following statements are true.
\begin{enumerate}
\item [(i)] The eigenvalues of $\widetilde{\rho}(t_{\cu})$ are precisely given by the set
$$\left\{\zeta\lambda_{(i,k)}^{1/h_i}| 1\leq i\leq n_{\cu}\,,1\leq k\leq e_i\,, \zeta\in R_{h_i}\right\},$$
where $e_i$ is the number of distinct eigenvalues of $\rho(t_i),$ and $R_{h_i}$ is the set of all $h_i^{th}$ root of unity.
\item [(ii)] If there is a Jordan block of size $m(\lambda_{(i,k)})$ in $\rho(t_i)$ associated to the eigenvalue $\lambda_{(i,k)},$ then there is a Jordan block in $\widetilde{\rho}(t_{\cu})$ of size $$m\left(\zeta\lambda_{(i,k)}^{1/h_i}\right):=m(\lambda_{(i,k)}),~\forall ~ 1\leq i\leq n_{\cu}\,,1\leq k\leq e_i\,, \zeta\in R_{h_i}.$$
\end{enumerate}
\end{lem}
Before proving Lemma~\ref{cor:eigenvalues}, we need to discuss generalized permutation matrices. These matrices are square matrices where each row and column contain precisely one non-zero entry. Unlike classical permutation matrices where the non-zero entry must be $1$, in generalized permutation matrices, this entry can take any non-zero value. Any invertible generalized permutation matrix can be expressed as the product of an invertible diagonal matrix and a permutation matrix. In fact, the group of all such matrices in $\text{GL}_h(\R)$ is isomorophic to $S_h\rtimes \Delta_h(\R)$, where $S_h$ is the symmetric group, and $\Delta_h(\R)$ be the set of diagonal matrices in $\text{GL}_h(\R)$. The semidirect product above is defined by the natural action of $S_h$ on $\Delta_h(\R)$ by permuting coordinates. With this identification, or simply by matrix multiplications, the reader may note that
$$\begin{psmallmatrix}
    0 & & &  & a\\ 
    1 &  \ddots&  & &\\ 
     & \ddots &  & \ddots &\\
     &   & 1 &  &0\\ 
 \end{psmallmatrix}^h_{h \times h}=aI_{h},~\forall a \in \R\setminus \{0\}.$$
Similarly, we have the following for any non-zero matrix $A\in \text{GL}_m(\R)$.
\begin{equation}\label{eqn:matrix}
\begin{psmallmatrix}
    0 & & &  & A\\ 
    I_{m} &  \ddots&  & &\\ 
     & \ddots &  & \ddots &\\
     &   & I_m &  &0\\ 
 \end{psmallmatrix}^h_{hm \times hm}=AI_{hm}.\end{equation}

\begin{proof}[Proof of Lemma~\ref{cor:eigenvalues}]
Let $v_{(i,k)}$ be an eigenvector of $\rho(t_i)$ associated to the eigenvalue $\lambda_{(i,k)}.$ Then it follows from the description of the block form given in Section~\ref{sec:liftinduction} that there exists a set of vectors $\left\{v_{(i,k,\zeta)}| 1\leq k\leq e_i, \zeta\in R_{h_i}\right\}$, whose $i^{\text{th}}$ block is of the form
$\left(\lambda_{(i,k)}^{1-1/h_i}\zeta^{-1} v_{(i,k)},\lambda_{(i,k)}^{1-2/h_i}\zeta^{-2}v_{(i,k)},\cdots, v_{(i,k)}\right)^t$, and all other entries are zero, satisfying that
\begin{equation*}\label{eqn:m}
\widetilde{\rho}(t_{\cu})(v_{(i,k,\zeta)})=\zeta\lambda_{(i,k)}^{1/h_i}v_{(i,k,\zeta)},~\forall~1\leq i\leq n_{\cu},1\leq k\leq e_i,\zeta \in R_{h_i}.
\end{equation*}
This shows that the set of eigenvalues of $\widetilde{\rho}(t_{\cu})$ contains $$\left\{\zeta\lambda_{(i,k)}^{1/h_i}| 1\leq i\leq n_{\cu}\,,1\leq k\leq e_i\,, \zeta\in R_{h_i}\right\}.$$ On the other hand, let us recall from Section~\ref{sec:liftinduction} that $\widetilde{\rho}(t_{\cu})$ is in a block diagonal form, where each block is of the form
 $$\begin{psmallmatrix}
    0 & & &  & \rho(t_i)\\ 
    I_m &  \ddots&  & &\\ 
     & \ddots &  & \ddots &\\
     &   & I_m &  &0\\ 
 \end{psmallmatrix}_{mh_i\times mh_i}.$$
If $\lambda'$ is any eigenvalue of such a block, we can derive from~(\ref{eqn:matrix}) that $(\lambda')^{h_i}=\lambda_{(i,k)}$ for some $1\leq k\leq e_i$. This provides the comprehensive description of the eigenvalues of $\widetilde{\rho}(t_{\cu})$, as intended.

To prove $(ii)$, considering $h=h_1h_2\cdots h_{n_{\cu}}$. From (\ref{eqn:matrix}), it follows that $\widetilde{\rho}(t_{\cu}^h)$ is in block diagonal form, where each block is precisely $\rho(t_i)^{h/h_i}$ for all $1\leq i\leq n_{\cu}$. Additionally, each such block appears $h_i$ times.

On the other hand, it is a known result that for any non-trivial Jordan block $J$, and for any integer $h\neq 0$, $J^h$ is also a non-trivial Jordan block of the same size as $J$.\footnote{This is a fact from linear algebra. For proof, the reader may refer to the answer given by Oscar Cunningham at https://math.stackexchange.com/questions/1849839/jordan-form-of-a-power-of-jordan-block.} This completes the proof.
\end{proof}
We now have all the required tools to prove the main result of this article.
\subsection{Proof of Theorem~\ref{thm:main}} Part $(a)$ is a consequence of part $(i)$ of Lemma~\ref{lem:cuspidal} and part $(i)$ of Proposition~\ref{prop:logarithmic}. Part $(b)$ follows from part $(ii)$ of Proposition~\ref{prop:logarithmic}.
\subsection{Space of the lifted vector-valued automorphic forms}
Let us consider $R_{\rG}$ to be the ring of weakly holomorphic scalar-valued automorphic forms associated with $\rG.$ Define $\m^{!}_{k}(\rho)$ (resp. $\m_{k,\text{hol}}(\rho)$ and $\m_{k,\text{cusp}}(\rho)$) be the set of all weakly holomorphic logarithmic vvaf (resp. holomorphic logarithmic vector-valued automorphic forms and logarithmic vector-valued cusp forms) of weight $k,$ associated to representation $\rho$ of $\rG.$ It is clear that $R_{\rG}$ acts on those spaces by means of the map $(f,\X)\mapsto f\X,$ for any $f\in R_{\rG}$ and $\X(\tau)$ from one of those spaces. As a consequence of the results proved earlier, we have the following consequence.
\begin{prop}\label{prop:isomorphism} Let $\rG$, $\mr{H}$, $\rho$ be as in Theorem~\ref{thm:main} and $k$ be an integer. Then there is a $\mr{R}_{\rG}$-module isomorphism between the following spaces.
\begin{enumerate}
\item[(i)] $\m^{!}_{k}(\rho)~\mathrm{and}~ \m^{!}_{k}(\widetilde{\rho}),$
\item[(ii)] $\m_{0,\mathrm{hol}}(\rho)~\mathrm{and} ~\m_{0,\mathrm{hol}}(\widetilde{\rho}),$ and
\item[(iii)] $\m_{0,\mathrm{cusp}}(\rho)~\mathrm{and}~\m_{0,\mathrm{cusp}}(\widetilde{\rho}).$
\end{enumerate}
\end{prop}
\begin{proof} Given a weakly holomorphic vvaf $\X(\tau)$ from one of the $\m^{!}_{k}(\rho), \m_{0,\mathrm{hol}}(\rho)$ or $\m_{0,\mathrm{cusp}}(\rho),$ we consider its lift $\widetilde{\X}(\tau)$. It follows from Theorem~\ref{thm:main} and part $(ii)$ of Lemma~\ref{lem:cuspidal} that, indeed, this map is a possible candidate for the isomorphism in any of the parts $(i), (ii)$ or $(iii).$ Clearly this is an injection. On the other, consider a vvaf $\Y:=(\Y_1,\Y_2,\cdots,\Y_d)$ in one of the spaces $\m^{!}_{k}(\widetilde{\rho}), \m_{0,\mathrm{hol}}(\widetilde{\rho})$ or $\m_{0,\mathrm{cusp}}(\widetilde{\rho}).$ Recall that the weight $0$ form $\Y_0(\tau):= \Delta_{\rG}^{-k/w_{\rG}}(\tau)\Y(\tau),$ obtained from $\Y(\tau).$ Without loss of generality, we may assume that the weight of $\Y(\tau)$ is $0,$ if not, we can simply work with $\Y_0(\tau).$

Let us now fix an element $\gamma\in \rG.$ Then for each $i,$ there exists a unique $j:=j(i)$ such that $\Y_i(\gamma\tau)=\rho(\gamma_i^{-1}\gamma\gamma_j)\Y_j(\tau).$ Moreover, any such $\gamma_i^{-1}\gamma\gamma_j$ is in $\rH.$ In particular, if $\gamma_1=1$ and $\gamma\in \rH$, then $\gamma_{j(1)}=1$ and we have that $\Y_1(\gamma\tau)=\rho(\gamma)\Y(\tau).$ This implies that, $\Y_1 \in 
\m^{!}_{k}(\rho), \m_{0,\mathrm{hol}}(\rho)$ or $\m_{0,\mathrm{cusp}}(\rho).$ On the other hand if we take $\gamma=\gamma_i,$ then we must have that $j(i)=1,$ and hence $\Y_i(\tau)=\Y_1(\gamma_i^{-1}(\tau)).$ Therefore, $\Y(\tau)=\widetilde{\Y_1}(\tau).$ This gives the required isomorphism between the given spaces.
\end{proof}

\section{Existence of logarithmic vector-valued automorphic forms}\label{sec:exist} In this section, we shall discuss the existence of vvaf associated with a given representation. To begin with, let us first recall~\cite[Section 4]{Bajpai2019} the scalar-valued cusp form $\Delta_{\rG}(\tau)$ of non-zero weight, which is holomorphic on $\mathbb{H}^{*}_{\rG}$ and nonzero everywhere, except at $\infty$. Moreover, let $\nu_{\rH}$ be the rank $1$ representation associated to $\Delta_{\rH}^{1/w_{\rH}}(\tau)$.
\begin{lem}\label{lem:exist}
Let $\rG$ be any Fuchsian group of the first kind, and $k$ be any integer. There exists a weakly holomorphic logarithmic vvaf of weight $k,$ associated with some representation of $\rG.$

\end{lem}
\begin{proof} Consider the representation $\mathbb{I}:\rG\to \text{GL}_2(\C),$ given by $\gamma \mapsto \gamma.$ Now consider the vector-valued holomorphic function $\X_0:\mathbb{H}\to \C^2$ given by $\tau \mapsto (\tau,1).$ Note that
$$\X_0\left(\frac{a\tau+b}{c\tau+d}\right)=(c\tau+d)^{-1}\X_0(\tau),~\forall~\tmt{a}{b}{c}{d}\in \rG,~\tau \in \mathbb{H}.$$
Of course, we can not conclude about the weight from here if $c=0$ and $d=1$ for every element in $\rG.$ However, that is not the case because $\rG$ is a Fuchsian group of the first kind. Therefore, $\X(\tau)$ is a weakly holomorphic logarithmic vvaf of weight $-1$ associated with $\mathbb{I}.$ Then $\X_k(\tau):=\X_0\Delta_{\rG}^{\frac{k+1}{w_{\rG}}}(\tau)$ is a weakly holomorphic logarithmic vvaf of weight $k$ associated with the representation $\rho:=\mathbb{I} \otimes \nu_{\rG}^{k+1}.$
\end{proof} 
Now one may naturally ask if given any representation $\rho:\rG \to \text{GL}_m(\mathbb{C})$ can we find a non-trivial weakly-holomorphic vvaf associated to $\rho$ whose components are linearly independent? Of course, the interesting question is when the components are linearly independent; otherwise, one could take a zero function. The first author confirmed this for admissible representations having a finite image, and in~\cite[Chapter 5]{BajpaiThesis}, the same is confirmed for any admissible representations with rational exponents\footnote{In other words, all the eigenvalues of the image of any parabolic element is rational.}, associated with any genus zero Fucshian groups of the first kind. Furthermore, in~\cite[Chapter 8]{BajpaiThesis}, explicit generators are produced for the rank $2$ case when the Fuchsian group is a modular triangle group. An extension to rank $3$ for the modular group is done by Candelori et. al. in~\cite{CHMY}.

The main goal of this section is to study this for logarithmic vector-valued automorphic forms. More precisely, we consider the following question.
\begin{question}\label{qn:existence}
Let $\rG$ be any Fuchsian group of the first kind. Given a logarithmic representation $\rho$ of $\rG$, can we find a non-trivial logarithmic vvaf associated with $\rho$ whose components are linearly independent? 
\end{question}
For the logarithmic representations associated with the modular group, the existence could be deduced from~\cite[Theorem 3.1]{Gannon1}. Gannon remarked that a similar result could be deduced for any genus $0$ Fuchsian group of the first kind. 

Before proving our main result of this section, let us quickly recall the result of~\cite{Bajpai2019}. As mentioned, the author assumed that the image of $\rho$ has a finite image and all the eigenvalues of $\rho(\gamma)$ are unitary for any parabolic element $\gamma \in \rG.$ In this case we must have that, $\mu(\cdot)$ of all such eigenvalues are rational.

Now for the logarithmic case, we do not have the luxury to assume that the representation has a finite image. If $\rho$ is not admissible, then there exists some parabolic element $\gamma\in \rG$ such that $\rho(\gamma)$ has at least one non-trivial Jordan block, say $J_{\lambda}.$ Suppose this block is of size $m(\lambda)\times m(\lambda).$ Since it is non-trivial, we may assume that $1<m(\lambda)\leq m.$ If $\rho$ has finite image, then the set $\{J_{\lambda}^n\}_{n\in \mathbb{N}}$ must also be finite. Note that,
$$J_{\lambda}^{n}=\lambda^{n}(I_{m(\lambda)}+N)^{n}=\lambda^n\sum_{0\leq i\leq m(\lambda)}\binom{n}{i}N^i,$$
where $N$ is the corresponding nilpotent matrix. In particular, for any large $n$, we note that there is an entry in $J_{\lambda}^n$ of modulus $\gg \binom{n}{m(\lambda)}.$ Therefore, we can not assume that the representation has a finite image for the logarithmic case. Here we would like to give a partial answer to Question~\ref{qn:existence} in the following form.

\begin{prop}\label{prop:existence} Let $\rH, \rG$ and $\rho$ be as in Theorem~\ref{thm:main}. Suppose that $\psi$ be a representation associated to $\rG$ such that $\widetilde{\rho}=\psi$. Then there exists a weakly holomorphic logarithmic vvaf $\widetilde{\X}(\tau)$ associated to $(\psi,\rG)$ with linearly independent components if and only if there exists a weakly holomorphic logarithmic vvaf $\X(\tau)$ associated to $(\rho,\rH)$ with linearly independent components.
\end{prop}

\begin{proof} It is enough to work with only weight $0$ forms because multiplication by a scalar-valued non-zero function does not change the linear independence. Suppose that there exists $(\rho,\rH)$ with an associated weakly holomorphic logarithmic vvaf $\X(\tau).$ It follows from Theorem~\ref{thm:main} that $\widetilde{\X}(\tau)$ is a weakly holomorphic logarithmic vvaf for $\text{Ind}_{\rH}^{\rG}(\rho)=\psi.$ We are now done for the if part, provided that the set $\left\{\X(g_{i,j}^{-1}\tau)| 1\leq i\leq n_{\infty}, 0\leq j<h_i\right\}$ is a linearly independent set. If not, we fix a point $\tau_0\in \mathbb{H}$ such that $g_{i,j}^{-1}\tau_0$ are pairwise distinct, and denote $m_{i,j}=\text{ord}_{\tau=\tau_0}(\X(g_{i,j}^{-1}(\tau))$. Let $f$ be any non-constant modular function for $\rH$. For a set of positive integers $\{n_{i,j}\}$ to be specified later, let us consider the function,
$$F(\tau)=\prod_{i=1}^{n_{\infty}}\prod_{j=0}^{h_i -1}\left(f(\tau)-f(g_{i,j}^{-1}\tau_0)\right)^{n_{i,j}}.$$
Note that $F(\tau)$ is a modular function for $\rH$, with 
$$\text{ord}_{\tau=\tau_0}F(g_{i,j}^{-1}\tau)=\sum n_{i,j}\text{ord}_{\tau=\tau_0}f(g_{i,j}^{-1}\tau).$$ Now we choose $\{n_{i,j}\}$ in such a way, so that $m_{i,j}+\text{ord}_{\tau=\tau_0}F(g_{i,j}^{-1}\tau)$ are pairwise distinct, and then $\text{ord}_{\tau=\tau_0}F\X(g_{i,j}^{-1}\tau)$ are pairwise distinct. Consequently, the set $\left\{F\X(g_{i,j}^{-1}\tau) | 1\leq i\leq n_{\infty}, 0\leq j<h_i\right\}$ is linearly independent. In particular, now $F\X(\tau)$ is a weakly holomorphic logarithmic vvaf for $\rH$ associated to $\rho$ and $\widetilde{\X}(\tau):=\widetilde{F\X}(\tau)$ is a weakly holomorphic logarithmic vvaf associated to $\widetilde{\rho}:=\psi$, with linearly independent components, as desired.

On the other hand, let $\widetilde{\X}(\tau)$ be a weakly holomorphic logarithmic vvaf for $(\psi,\rG)$ with linearly independent coefficients. In this case, the proof immediately follows from part $(i)$ of Proposition~\ref{prop:isomorphism}.
\end{proof}
We record the following consequence by combining Lemma~\ref{lem:exist} and Proposition~\ref{prop:existence}.
\begin{cor}\label{cor:existence} 
Assume that $\rho$ lies in the following set of representations,
$$\left\{\widetilde{\mathbb{I}}\otimes v_{\rG}^k |~k\in \Z,~\rH \subseteq \rG,~ [\rG/\rH]<\infty \right\}.$$
Then we may find a logarithmic vvaf associated with $\rho$, whose components are linearly independent.
\end{cor}

\section{Fourier coefficients of the lifted forms}\label{sec:computation}
In the previous section, we discussed the cuspidal properties of the lifted automorphic forms. In this section, we shall study their Fourier coefficients and prove Theorem~\ref{thm:relation}. 

\subsection{On the growth of the Fourier coefficients}In~\cite{BBF}, the authors studied the growth of any holomorphic vvaf associated with any Fuchsian group of the first kind $\bar{\rH}$. More precisely, they showed that there exists a constant $\alpha$ (depending on the associated representation) such that $\norm{\X[n]}\ll_{\bar{\rH},\rho} n^{k+2\alpha},$ where $\X(\tau)$ is a holomorphic vvaf of weight $k\in 2\mathbb{Z}$ associated a representation $\rho$ of $\bar{\rH}.$ Moreover, the constant $\alpha$ depends only on $\bar{\rH}$ and the exponent $k+2\alpha$ can be divided by $2$ for cuspforms. In this section, we shall not only work with $\bar{\rH}$ but also with $\rH$, and study the change of this exponent under lifting.

In general, we show that the constant $\alpha$ is multiplied by, at most, the index of $\rH$ in $\rG.$ In particular, the exponent does not change when $\alpha=0.$ In~\cite{BBF}, the authors remarked that $\alpha$ can be taken to be $0$ when $\rho$ is a unitary representation. To prove the main result of this section, let us start with the following.
\begin{lem}\label{lem:growth} Let $\rho$ be a representation of $\rH$ such that $\norm{\rho(h)}\ll_{\rH} \norm{h}^{\alpha},~\forall h\in\rH.$ Then the induced representation $\widetilde{\rho}$ associated to a finite extension $\rG$ of $\rH$ has the following growth.
$$\norm{\widetilde{\rho}(\gamma)}\ll_{\rho,\rG} \norm{\gamma}^{\alpha},~\forall \gamma\in \rG.$$
\end{lem}
\begin{proof} It follows from the definition of induced representation that,
$$\norm{\widetilde{\rho}(\gamma)}\leq \max_{\substack{1\leq i,j\leq d\\ \gamma_i\gamma\gamma_j^{-1}\in \rH}}\norm{\rho(\gamma_i\gamma\gamma_j^{-1})}.$$
On the other hand, it follows from the assumption on $\rho,$ and the semi multiplicative property of $\norm{\cdot}$ that $\norm{\rho(\gamma_i\gamma\gamma_j^{-1})}\ll_{\rho,\rG}\norm{\gamma}^{\alpha}.$ 
\end{proof}
Let us now recall from Lemma~\ref{lem:cuspidal} that if $\X(\tau)$ is a holomorphic vvaf of weight $0,$ then the lifted form $\widetilde{\X}(\tau)$ is a holomorphic vvaf as well. As a consequence of Lemma~\ref{lem:growth}, we deduce the following.
\begin{cor} Let $\rH,\rG$ and $\rho$ be as in Theorem~\ref{thm:main}, and $\X(\tau)$ be an admissible holomorphic vvaf associated to $(\rH,\rho)$ of weight $0.$ Then there exists a constant $\alpha$ depending only on $\rH$ and $\rho$ such that the Fourier coefficients of the lifted holomorphic vvaf $\widetilde{\X}(\tau)$ have the following growth
$$\norm{\widetilde{\X}[n]}\ll n^{\alpha}.$$
In particular, the growth of the Fourier coefficients of the lifted vector-valued automorphic forms does not depend on $\rG.$ However, if $\X(\tau)$ is a holomorphic logarithmic vvaf, then the exponent increases by at most $\mathrm{rank}(\widetilde{\rho}):=\mathrm{rank}(\rho)|\rG/\rH|.$ 
 
\end{cor}
\begin{proof}When $\rho$ is admissible, it follows from~\cite[Lemma 4.1]{BBF} that $$\norm{\widetilde{\rho}(h)}\ll_{\rH} \norm{h}^{\alpha},~\forall h\in \rH,$$ 
for some constant $\alpha$ depending on $\rH$ and $\rho$. In particular, Lemma~\ref{lem:growth} implies that
$$\norm{\widetilde{\rho}(\gamma)}\ll_{\rG} \norm{\gamma}^{\alpha}.$$
Since $\X(\tau)$ has even integer weight, it is evident that $\rho(I)=\rho(-I).$ Therefore $\rho$ can be thought of as a representation of $\rH$. Proof for the admissible case then follows from~\cite[Theorem 1.1]{BBF}.

For the logarithmic case, one can essentially follow the argument on page 21 of~\cite{BBF}. The increase in the exponent is coming because of the extra logarithmic terms in the Fourier expansion, and they come with power at most $\text{rank}(\widetilde{\rho}),$ which is precisely $\text{rank}(\rho)|\rG/\rH|=md.$
\end{proof}
\subsection{Computations of the Fourier coefficients} In this section, we shall study the relationship between the Fourier coefficients of a weakly holomorphic vvaf and its lift. Consequently, we shall draw some conclusions on the unbounded denominator problem.

Let us start by recalling the coset representatives $\left\{g_{i,j}| 1\leq i\leq n_{\infty}, 0\leq j< h_i\right\}$ of $\rH$ in $\rG$ from Section~\ref{sec:lift}. Let $\X(\tau)$ be a weakly holomorphic vvaf of weight $0$, then by definition 
$$\widetilde{\X}(\tau)=\left(\X(g_{i,j}^{-1}\tau)\right)^t_{\substack{1\leq i\leq n_{\infty}\\ 0\leq j< h_i}}.$$ To get the Fourier coefficients of $\widetilde{\X}(\tau)$, say, at the cusp $\infty,$ we need to study the Fourier expansion of each component $\X(g_{i,j}^{-1}\tau)$ at $\infty.$ 

Let $\rH, \rG$ be as in Theorem~\ref{thm:main} and $\rho$ be an admissible representation of $\rH$. Let $\infty$ be a cusp of $\rG$ and $\left\{\cu_{i} | 1\leq i\leq n_{\infty}\right\}$ be the set of all inequivalent cusps of $\rH$ lying under $\infty.$ For each such $i,$ we write $\rho(t_i)=P_iJ_iP_i^{-1},$ and $\widetilde{\rho}(t_{\infty})=P_{\infty}J_{\infty}P_{\infty}^{-1},$ where $J_i$ and $J_{\infty}$ are the corresponding diagonal matrices. It follows from the proof of Theorem~4.2 in~\cite{Bajpai2019} that we can write $P_{\infty}$ in a block diagonal form $\text{diag}(P'_1,P'_2,\cdots P'_{n_{\infty}})$, where each block $P'_i$ is of the size $mh_i\times mh_i,$ which can be described by the proof of Theorem 4.2 in~\cite{Bajpai2019}. In particular, when $P_i$ are all trivial, then all the entries of any such block $P'_i$ either $0$ or $\lambda^{1-j/h_i}\zeta^j,$ for some $\zeta \in R_{h_i}.$

When $\rho$ is a logarithmic representation, $J_i$ are in a Jordan normal form. Even in this case, we can write each $P_i$ in a block diagonal form as before. This follows from the proof of Lemma~\ref{cor:eigenvalues}.

\subsection{Proof of Theorem~\ref{thm:relation}}
\subsection*{Admissible case.}
In this case, $\widetilde{\X}[n]$ is denoted to be the $n^{\text{th}}$ Fourier coefficient of $\widetilde{\X}(\tau)$ at $\infty,$ which is a $m d\times 1$ vector. For any cusp $\cu_i$ of $\rH$ lying under $\infty,$ write $\rho(t_{i})=P_i e^{2\pi\Ii\Lambda_{\cu_i}}P^{-1}_i$ and $$\widetilde{\rho}(t_{\infty})=P_{\infty}\mathrm{diag}\left(e^{2\pi\Ii \frac{\Lambda_{\cu_i}+j}{h_i}}\right)_{\substack{1\leq i\leq n_{\infty},\\0\leq j<h_i}}P^{-1}_{\infty}.$$ It follows from Corollary 4.7 of~\cite{Bajpai2019}, equation~(\ref{eqn:admf}), and the discussion in the previous paragraph that, for each $1\leq i\leq n_{\infty},$ we can write 
$$\left(\X(g^{-1}_{i,j}\tau)\right)^t_{0\leq j<h_i}={P'}_{i} \mathrm{diag}\left(\widetilde{q}^{\frac{\Lambda_{\cu_i}+j}{h_i}}\right)_{0\leq j<h_i} {P'}_{i}^{-1}\left(\sum_{n=-M_{i,j'}}^{\infty}\widetilde{\X}_{i,j'}[n]\widetilde{q}^{n}\right)^t_{0\leq j'<h_i},$$
where each $\widetilde{\X}_{i,j'}[n]$ is a $m\times 1$ vector. 

Denoting $\left({\widetilde{\X}'}_{i,j}[n]\right)^t_{0\leq j<h_i}={P'}_i^{-1}\left(\widetilde{\X}_{i,j'}[n]\right)^t_{0\leq j'<h_i}$, and using the identity that $h_{\cu_i}=h_ih_{\infty},$ we have
\begin{align}\label{algn10}
{P'}_{i}^{-1}\left(\X(g^{-1}_{i,j}\tau)\right)^t_{0\leq j<h_i}&=\left(\widetilde{q}^{\frac{\Lambda_{\cu_i}+j}{h_i}}\sum_{n=-M_{i,j}}^{\infty} {\widetilde{\X}'}_{i,j}[n]\widetilde{q}^{n}\right)^t_{0\leq j<h_i}\\\nonumber 
&= \left(e^{\frac{2\pi \Ii \tau}{h_{\infty}}\frac{\Lambda_{\cu_i}+j}{h_i}}\sum_{n=-M_{i,j}}^{\infty} {\widetilde{\X}'}_{i,j}[n]e^{\frac{2n\pi \Ii \tau}{h_{\infty}}}\right)^t_{0\leq j<h_i}\\\nonumber 
&=\left(e^{\frac{2\pi \Ii \tau}{h_{\cu_i}}\Lambda_{\cu_i}}\sum_{n=-M_{i,j}}^{\infty} {\widetilde{\X}'}_{i,j}[n]e^{\frac{2(nh_i+j)\pi \Ii \tau}{h_{\cu_i}}}\right)^t_{0\leq j<h_i}.
\end{align}
On the other hand, $\X|_kg_{i,j}^{-1}(\tau):=\X|_0g_{i,j}^{-1}(\tau)=\X(g_{i,j}^{-1}\tau)$ is a weakly holomorphic vvaf associated to $\rH_{i,j}:=g_{i,j}\rH g_{i,j}^{-1},$ with respect to the representation $\rho_{i,j}:=g_{i,j}\gamma g_{i,j}^{-1} \mapsto \gamma.$ Moreover, $\infty$ is a cusp of $\rH_{i,j}$ because $g_{i,j}^{-1}(\infty)=\cu_i,$ and $\cu_i$ is a cusp of $\rH.$ In other words, $g_{i,j}t_{i} g_{i,j}^{-1}$ is the generator of the stabilizer subgroup of $\infty$ in $\rH_{i,j},$ with $\rho_{i,j}(g_{i,j}t_{i} g_{i,j}^{-1})=\rho(t_{i})$, for any pair $(i,j)$. Therefore, we can write
\begin{align}\label{algn11}
P_i^{-1}\X(g_{i,j}^{-1}\tau)&=e^{\frac{2\pi \Ii \tau}{h_{\cu_i}}\Lambda_{\cu_i}}\sum_{n=-M'_{i,j}}^{\infty} \X'^{(i,j)}[n]e^{\frac{2n\pi \Ii \tau}{h_{\cu_i}}},~\forall j\leq h_i,
\end{align}
where ${\X'}^{(i,j)}[n]=P^{-1}_i\X^{(i,j)}[n]$. 

Combining (\ref{algn10}), and (\ref{algn11}), we have the following equality for each $1\leq i\leq n_{\infty}$
\begin{equation*}\label{eqn:semifinal}
\left(\sum_{n=-{M'}_{i,j}}^{\infty} {\X'}^{(i,j)}[n]e^{\frac{2\pi \Ii \tau}{h_{\cu_i}}(n)}\right)^t_{0\leq j<h_i}=Q_i\left(\sum_{n=-M_{i,j'}}^{\infty} {\widetilde{\X}'}_{i,j'}[n]e^{\frac{2\pi \Ii \tau}{h_{\cu_i}}(nh_i+j')}\right)^t_{0\leq j'<h_i},
\end{equation*}
where $Q_i$ is the $mh_i\times mh_i$ matrix, given by $Q_i=P_i^{-1}P'_i$. 

Now, recalling that 
$${\X'}^{(i,j)}[n]=P^{-1}_i\X^{(i,j)}[n],~\left({\widetilde{\X}'}_{i,j}[n]\right)^t_{0\leq j<h_i}={P'}_i^{-1}\left(\widetilde{\X}_{i,j'}[n]\right)^t_{0\leq j'<h_i},$$
it finally follows from $(\ref{eqn:semifinal})$ that, for any $0\leq j<h_i$,
\begin{equation}\label{eqn:admfinal}
   \widetilde{\X}_{i,j}[n]=
    \X^{(i,j)}\big[nh_i+j\big].
\end{equation}

\subsection*{Logarithmic case.} Suppose that $\X(\tau)$ is a weakly holomorphic logarithmic vvaf, i.e., let $\rho$ be a logarithmic representation of $\rH.$ Let $\{\lambda_{(i,k)}|1\leq k\leq e_i\}$ be the set of eigenvalues of $\rho(t_i)$ for each $1\leq i\leq n_{\infty},$ where $e_i$ is number of distinct eigenvalues of $\rho(t_i)$. From Lemma~\ref{cor:eigenvalues}, we have an explicit description of the eigenvalues (with multiplicity) of $\rho(t_{\infty}),$ in terms of the $\lambda_{(i,k)}$. 

The argument to prove Theorem~\ref{thm:relation} for this logarithmic case is essentially the same as in the admissible case; therefore, we shall omit some details. Following~(\ref{eqn:logarithmic}), and recalling the identity from Lemma~\ref{cor:eigenvalues} that $m(\lambda)=m(\zeta \lambda^{1/h_i}),$ we can write the following for each $1\leq i\leq n_{\infty}$

\resizebox{\textwidth}{!}{
\parbox{\textwidth}{
\begin{align}\label{algn:4}
\left(\X(g_{i,j}^{-1}\tau)\right)^t_{0\leq j<h_i}
&= P'_i \sum_{\zeta \in R_{h_i}}\sum_{\lambda\in R(t_i)}\sum_{j'=0}^{m(\lambda)-1}\left(\log \widetilde{q}\right)^{j'}h_i^{j'}\widetilde{q}^{\mu(\lambda)}\sum_{n=-M_{\lambda,j'}}^\infty\widetilde{\X}'[\zeta \lambda^{1/h_i},j',n] \widetilde{q}^{nh_i+j(\zeta)},
\end{align}
}}
where we denote ${\widetilde{\X}'}[\zeta\lambda^{1/h_i},j',n]={P'}_{i}^{-1} \widetilde{\X}[\zeta\lambda^{1/h_i},j',n]$, $j(\zeta)$ to be the numerator of $\mu(\zeta)$, and $\widetilde{q}=e^{\frac{2\pi \Ii \tau}{h_{\infty}}}$, as in~$(\ref{algn10})$.

On the other hand, realizing $\X(g_{{i,j}}^{-1}\tau)$ as a weakly holomorphic vvaf associated to $H_{i,j}$, we can write the following for each $0\leq j< h_i$
\begin{align}\label{algn:5}
\X(g_{i,j}^{-1}\tau)&=P_i\sum_{\lambda\in R(t_{i})}\sum_{j'=0}^{m(\lambda)-1}(\log \widetilde{q})^{j'}\widetilde{q}^{\mu(\lambda)}\sum_{n=-M'_{\lambda,j'}}^\infty\X'^{(i,j)}[\lambda,j',n] \widetilde{q}^n,
\end{align}
where we denote $\X'^{(i,j)}[\lambda,j',n]=P_i^{-1}\X^{(i,j)}[\lambda,j',n]$, and $\widetilde{q}=e^{\frac{2\pi \Ii \tau}{h_{\cu_i}}}$, as in~$(\ref{algn11})$.

Comparing (\ref{algn:4}) and (\ref{algn:5}) in a similar manner as in the admissible case, we have $h_iM_{\lambda,j'}=M'_{\lambda,j'}$. Furthermore, for any positive integer $n$ and any $j'$ such that $0\leq j'<m(\lambda)$, we have
\begin{equation}\label{eqn:logfinal}
   \widetilde{\X}[\zeta \lambda^{1/h_i},j',n]= 
    h_i^{-j'}\X^{(i,j)}\big[\lambda,j',nh_i+j\big],
    \end{equation}
where $j=j(\zeta)$.

Therefore, the proof of both admissible and logarithmic cases follows by setting all the $r_i,m_i$, and $n_i$ appropriately as per the requirements in~(\ref{eqn:admfinal}) and~(\ref{eqn:logfinal}).
\qed
\newpage
\begin{remark}\rm $\;$
\begin{enumerate}
\item[(a)] One can also recover Lemma~\ref{lem:cuspidal} from Theorem~\ref{thm:relation}.
\item[(b)] More generally let $k$ be an arbitrary integer and $\X(\tau)$ be a weakly holomorphic vvaf of weight $k.$ Recall that, $\widetilde{\X}(\tau):=\widetilde{\X_0}^{(\infty)}(\tau)\Delta_{\rG}^{k/w_{\rG}}(\tau),$ where $\X_0(\tau)=\Delta_{\rG}^{-k/w_{\rG}}(\tau)\X(\tau).$ In the proof of Theorem~\ref{thm:relation}, we basically saw how to compute the Fourier coefficients of $\X_0(\tau).$ By the same argument, the Fourier coefficients of $\widetilde{\X}(\tau)$ can be written in terms of the coefficients of $\{\X_0g_{i,j}^{-1}(\tau)\}_{\substack{1\leq i\leq n_{n_{\infty}}\\ 0\leq j<h_i}}$, and $\Delta_{\rG}(\tau).$
\item[(c)] In certain cases, it turns out that $A_{i}=A_{\cu_i}.$ In those cases, we look at the expansion at the cusp $\cu_i$ while considering the vvaf $\X|_kA_{\cu_i}.$ For example, if $\rG=\Gamma(1)$ and $\rH=\Gamma_0(2),$ then one of the cusp is $0$ and one of the $A_i$ is precisely $S.$ However, in general, it may not always be possible that any $A_{\cu_i}$ is in $\rG.$ In those cases, they differ by an element of the form $\tmt{a}{\alpha}{0}{1/a},$ as already explained in the proof of Lemma~\ref{lem:gtu}.
\end{enumerate}
\end{remark}

\section*{Acknowledgements}
The authors would like to thank for their hospitality the Georg-August Universit\"at G\"ottingen, Max Planck Institut f\"ur Mathematik, Bonn, and IISER TVM, India, where much of the work on this article was accomplished. SB is supported by ERC Consolidator grant 648329 (codename GRANT, with H. Helfgott as PI), and the fellowship of IISER TVM.


\nocite{}
 \bibliographystyle{abbrv}
 \bibliography{arxiv}

\begin{thebibliography}{1}

\bibitem{BajpaiThesis}
J.~Bajpai.
\newblock {\em On vector valued automorphic forms}.
\newblock Ph.D. Thesis (Advisor: Terry Gannon), University of Alberta, 2015.

\bibitem{Bajpai2019}
J.~Bajpai.
\newblock Lifting of modular forms.
\newblock {\em Publications Math\'ematiques de Besan\c con}, (1):5--20, 2019.

\bibitem{BBF}
J.~Bajpai, S.~Bhakta, and R.~Finder.
\newblock Growth of {F}ourier coefficients of vector-valued automorphic forms.
\newblock {\em J. Number Theory}, 249:237--273, 2023.

\bibitem{CHMY}
L.~Candelori, T.~Hartland, C.~Marks, and D.~Y\'{e}pez.
\newblock Indecomposable vector-valued modular forms and periods of modular curves.
\newblock {\em Res. Number Theory}, 4(2):Paper No. 17, 24, 2018.

\bibitem{Gannon1}
T.~Gannon.
\newblock The theory of vector-valued modular forms for the modular group.
\newblock In {\em Conformal field theory, automorphic forms and related topics}, volume~8 of {\em Contrib. Math. Comput. Sci.}, pages 247--286. Springer, Heidelberg, 2014.

\bibitem{KM2011}
M.~Knopp and G.~Mason.
\newblock Logarithmic vector-valued modular forms.
\newblock {\em Acta Arith.}, 147(3):261--282, 2011.

\bibitem{KM2012}
M.~Knopp and G.~Mason.
\newblock Logarithmic vector-valued modular forms and polynomial-growth estimates of their {F}ourier coefficients.
\newblock {\em Ramanujan J.}, 29(1-3):213--223, 2012.

\end{thebibliography}
\end{document}